\documentclass[12pt, reqno]{amsart}
\usepackage{amsmath, amsthm, amscd, amsfonts, amssymb, graphicx, color}
\usepackage[bookmarksnumbered, colorlinks, plainpages]{hyperref}
\usepackage{mathrsfs}  
\newtheorem{theorem}{Theorem}[section]

\newtheorem{corollary}[theorem]{Corollary}
\theoremstyle{definition}

\theoremstyle{remark}

\numberwithin{equation}{section}

\begin{document}
\setcounter{page}{1}

%{\small Annals of Mathematics and Computer Science}

%{\small Vol 1 (2021) 1-10}

\centerline{}

\centerline{}

\title[The Rainbow Connection Number of The Commuting Graph]{The Rainbow Connection Number of The Commuting Graph of a Finite Non-abelian Group}

\author[R.F. Umbara, A.N.M. Salman, P.E. Putri]{Rian Febrian Umbara${^1}$$^{,3}$, A.N.M. Salman$^{2}$$^{\dagger}$, Pritta Etriana Putri$^{2}$$^{\ddagger}$}

\address{$^{1}$  Doctoral Program in Mathematics, Faculty of Mathematics and Natural Sciences, Institut Teknologi Bandung, Bandung, Indonesia}
\email{\textcolor[rgb]{0.00,0.00,0.84}{30119014@mahasiswa.itb.ac.id}}

\address{$^{2}$ Combinatorial Mathematics Research Group, Faculty of Mathematics and Natural Sciences, Institut Teknologi Bandung, Indonesia}
\email{$^{\dagger}$\textcolor[rgb]{0.00,0.00,0.84}{msalman@itb.ac.id},$^{\ddagger}$\textcolor[rgb]{0.00,0.00,0.84}{etrianaputri@itb.ac.id}}

\address{$^{3}$  School of Computing, Telkom University, Bandung, Indonesia}
\email{\textcolor[rgb]{0.00,0.00,0.84}{rianum@telkomuniversity.ac.id}}

%\dedicatory{This paper is dedicated to Professor ABCD}

\subjclass[2010]{Primary 05C15; Secondary 05C25.}

\keywords{rainbow connection number, commuting graph, finite nonabelian group}

%\date{Received: xxxxxx; Revised: yyyyyy; Accepted: zzzzzz.
%\newline \indent $^{*}$ Corresponding author}

\begin{abstract}
A path in an edge-colored graph is called a rainbow path if every two distinct edges of the path have different colors. A graph whose every pair of vertices are linked by a rainbow path is called a rainbow-connected graph. The rainbow connection number of a graph is the minimum number of colors that are needed to color the edges of the graph such that the graph is rainbow-connected. In this paper, we determine the rainbow connection number of the commuting graph of a finite non-abelian group, which is a graph whose vertex set is a non-abelian group of finite order and two distinct elements of the group is adjacent in the graph if they commute. We also show that the rainbow connection number of the commuting graph of a finite group is related to the number of maximal abelian subgroups or the number of involutions that do not commute with any other non-identity element of the group. 
\end{abstract} \maketitle

\section{Introduction}
Let $G = (V(G), E(G))$ be a graph with $V(G)$ and $E(G)$ are its set of vertices and set of edges, respectively. Let $c$ be an edge coloring on $G$ which allows the adjacent edges of $G$ have the same color. A path $P$ betwen two vertices of $G$ is a \textit{rainbow path} under the coloring $c$ if every pair of edges of the path have different colors. If every two vertices of $G$ are connected by a raibow path, then $G$ is \textit{rainbow-conneted}. The minimum number of colors that are needed to color the edges of $G$ such that $G$ is rainbow-connected is called the \textit{rainbow connection number} of $G$, denoted by $rc(G)$. A rainbow coloring of $G$ that uses $rc(G)$ colors is called a \textit{minimum rainbow coloring} of $G$.

The study of rainbow connection number of a graph was initiated by Chartrand et al. in 2008~\cite{Char}. The concept of rainbow connection number was proposed to solve the problem of communication security between two agencies of US Government after the terrorist attack on September 11, 2001. The rainbow connection number of a graph serves as a metric to assess the security of a communication or computer network represented by the graph. Besides defining rainbow connection number of a graph, Chartrand et al. also determined the rainbow connection numbers of some graphs, such as complete graphs, trees, wheel graphs, bipartite graphs, and multipartite graphs~\cite{Char}. Other researchers have also studied the rainbow connection numbers of other graph families. Some of them are Fitriani et al. who studied the rainbow connection number of amalgamation of some graphs \cite{Fitri1} and comb product of graphs \cite{Fitri2} .

Over the last decade, some researchers have studied the rainbow connection numbers of some graphs of finite groups. A graph of a finite group is a graph whose vertex set is the elements of a finite group and two elements of the group are adjacent if they meet a certain condition. Some example of graphs of finite groups are Cayley graphs~\cite{Cay}, non-commuting graph of a group~\cite{Abdol}, commuting graph of a group~\cite{Ali}, undirected power graphs of semigroups~\cite{Chak}, the enhanced power graph of a group~\cite{Aalipour}, and the inverse graph of a finite group~\cite{Alfur}. Some studies have been conducted to study rainbow connection numbers of some graphs of finite groups, such as the (strong) rainbow connection numbers of Cayley graphs on abelian groups~\cite{Li}, rainbow 2-connection numbers of Cayley graphs~\cite{Lu}, rainbow connection number of the power graph of a finite group~\cite{Ma1}, rainbow connectivity of the non-commuting graph of a finite group~\cite{Wei}, rainbow connection numbers of Cayley graphs~\cite{Ma2}, rainbow connectivity in some Cayley graphs~\cite{Bau}, and rainbow connection number of the inverse graph of a finite group~(\cite{Umb},\cite{umb1}). 

Motivated by those studies, we study the rainbow connection number of the commuting graph of a finite nonabelian group. The commuting graph $CG(\Gamma, X)$, where $\Gamma$ is a finite group and $X$ is a nonempty subset of $\Gamma$, is a graph with $X$ as its set of vertices and $u,v\in X$ are adjacent if $u$ and $v$ commute in $\Gamma$~\cite{Ali}. If $X = \Gamma$, then $CG(\Gamma, X)$ is written as $CG(\Gamma)$. In this paper, we discuss $CG(\Gamma)$, where $\Gamma$ is a finite nonabelian group. We show that for a finite nonabelian group $\Gamma$ with a nontrivial center, the rainbow connection number of $CG(\Gamma)$ is related to the number of maximal abelian subgroups of $\Gamma$.  We also show that if $\Gamma$ is a finite nonabelian group with a trivial center, the rainbow connection number of $CG(\Gamma)$ determines the number of involutions of $\Gamma$ that do not commute with any other nonidentity element of $\Gamma$.

%This paper is organized as follows. In Section 2, we present some preliminaries that are needed to obtain the main results. In Section 3, we investigate the rainbow connection numbers of the commuting graphs of finite nonabelian groups. We also discuss the relation between the rainbow connection number of a finite nonabelian group with the structure of the group.
\section{Preliminaries}
\subsection{Some terminologies in Graph Theory}
We refer to \cite{Dies} for definitions and notations in Graph Theory that are not described in this text. A \textit{graph} is a pair $G=(V(G),E(G))$ of sets, where the elements of $E(G)$ are 2-element subsets of $V(G)$. An element of $V(G)$ is called a \textit{vertex} of $G$ and an element of $E(G)$ is called an \textit{edge} of $G$. A graph $G=(V(G),E(G))$ is \textit{finite} if $|V(G)|$ is finite. If $|V(G)|=0$ of $1$, then $G$ is called a \textit{trivial} graph. An edge $\epsilon=\{x,y\}$ is written as $\epsilon=xy$ or $\epsilon=yx$. Two vertices $x$ and $y$ of a graph $G$ are \textit{adjacent} if $xy$ is an edge of $G$. An edge $xy$ is \textit{incident} with both $x$ and $y$. Both $x$ and $y$ are the \textit{ends} of the edge $xy$. A vertex of a graph $G$ which is adjacent to one and only one vertex of $G$ is called a \textit{pendant} vertex. An edge that incidents with a pendant vertex is called a \textit{pendant edge}. If every two distinct vertices of a graph $G$ is adjacent, then $G$ is a \textit{complete} graph. Let $k\geq 1$ be an integer. A \textit{path} between two vertices $x_0$ and $x_k$ is a graph with a set of vertices $\{x_0, x_1, \dots, x_k\}$ and a set of edges $\{x_0x_1, x_1x_2, \dots, x_{k-1}x_k\}$, where $x_i\ne x_j$ for every distinct $i,j \in \{0,1,\dots,k\}$. If there is a path between every two distinct vertices of a graph, then the graph is \textit{connected}.  All graphs in this paper are nontrivial and connected. 

The \textit{length} of a path is the number of edges in the path. The \textit{distance} between two vertices $x$ and $y$, denoted by $d(x,y)$, in a graph $G$ is the length of the shortest path between $x$ and $y$ in $G$. The largest distance between any two vertices of a graph $G$ is called the \textit{diameter} of $G$, denoted by $diam(G)$. Let $r\geq2$ be an integer. A graph $G=(V(G),E(G))$ is called an \textit{$r$-partite} graph if $V(G)$ can be partitioned into $r$ classes such that every member of $E(G)$ has it ends in different classes.

%Define an edge coloring $c: E(G) \rightarrow \{1,2,\dots,r\}$, $r\in \mathbb{N}$, where $E(G)$ is the set of edges of $G$ and adjacent edges may have the same colors. A path in $G$ is called a \textit{rainbow path} if no two edges of the path have the same color \cite{Char}. If there is a rainbow path between every two vertices of $G$, then $G$ is \textit{rainbow-connected} (with respect to $c$). In this case, $c$ is called a \textit{rainbow coloring} of $G$. The minimum number of colors that are needed to color the edges of $G$ such that $G$ is rainbow-connected is called the \textit{rainbow connection number} of $G$, denoted by $rc(G)$ \cite{Char}. A rainbow coloring of $G$ that uses $rc(G)$ colors is called a \textit{minimum rainbow coloring} of $G$. 
\subsection{Some properties of the rainbow connection number of a graph}
Let $G$ be a connected graph with at least two vertices. Chartrand et al. \cite{Char} have determined some properties of the rainbow connection number of $G$. For a connected graph $G$, $diam(G) \leq rc(G) \leq m$, where $diam(G)$ is the diameter of $G$ and $m$ is the number of edges of $G$. For a natural number $t$, if a connected graph $G$ with $|G|\geq3$ has $t$ pendant vertices, then $rc(G)\geq t$. The rainbow connection number of a complete graph and a $k$-partite graph, where $k\geq3$, have also been determined. 

\begin{theorem}\label{Thm:1}\cite{Char}
	Let $G$ be a connected graph. The rainbow connection number of $G$ is 1 if and only if $G$ is a complete graph.
\end{theorem}  

\begin{theorem}\label{Thm:2} \cite{Char}
	Let $G=K_{n_1,n_2,\dots,n_k}$ be a complete $k$-partite graph, with $k\geq 3$ and $n_1\leq n_2 \leq \dots\leq n_k$, such that $s=\sum_{1}^{k-1} n_i$ and $t=n_k$. Then
	
	$rc(G)=
	\begin{cases}
		1 & \text{if } n_k = 1\\
		2 & \text{if } n_k\geq2, s>t \\
		min\{3,\lceil\sqrt[s]{t}\rceil\}& \text{if } s\leq t
	\end{cases}$
\end{theorem}  
\subsection{Group and the commuting graph of a group}
We refer to \cite{Dum} for definitions and notations in Group Theory that are not described in this paper. A \textit{group} is an ordered pair $(\Gamma,*)$, where $\Gamma$ is a set and $*$ is a binary operation, which satisfies the following axioms:
\begin{enumerate}
	\item $(a*b)*c = a*(b*c)$ for all $a,b,c \in \Gamma$,
	\item there exists an element $e\in \Gamma$, which is called the \textit{identity element} of $\Gamma$, such that $a*e = e*a = a$ for all $a\in \Gamma$,
	\item for each $a\in \Gamma$ there is an element $a^{-1}\in \Gamma$, which is called the \textit{inverse} of $a$, such that $a*a^{-1} = a^{-1}*a=e$.
\end{enumerate}

For simplicity, the group notation $(\Gamma,*)$ will be written as $\Gamma$. Let $\Gamma$ be a group. Two members $a,b \in \Gamma$ \textit{commute} if $a*b=b*a$. A group $\Gamma$ is called an \textit{abelian} group if every two members of $\Gamma$ commute. The \textit{center} of a group $\Gamma$ is $Z(\Gamma)=\{z\in \Gamma | z*a = a*z$ for every $a \in \Gamma\}$. If $Z(\Gamma)=\{e\}$, where $e$ is the identity element of $\Gamma$, then $Z(\Gamma)$ is called \textit{trivial}. A group $\Gamma$ is called a \textit{finite group} if $|\Gamma|$ (the cardinality of the set $\Gamma$) is finite. An element $a\in \Gamma$ is called an \textit{involution} if $a\ne e$ and $a*a=e$. If $\Gamma$ is a group and $\Lambda$ is a nonempty subset of $\Gamma$ such that for all $h,k \in \Lambda$, $h*k \in \Lambda$ and $h^{-1}\in \Lambda$, then $\Lambda$ is called a \textit{subgroup} of $\Gamma$. If $\Lambda$ is a subgroup of $\Gamma$ such that every two elements of $\Lambda$ commute, then $\Lambda$ is an \textit{abelian subgroup} of $\Gamma$. If $\Lambda$ is an \textit{abelian subgroup} of a group $\Gamma$ and there is no other abelian subgroup of $\Gamma$ that contains $\Lambda$, then $\Lambda$ is called a \textit{maximal abelian subgroup} of $\Gamma$. 

The identity element of a finite group is a member of every maximal abelian subgroup of the group. Hence, the cardinality of any maximal abelian subgroup of a finite nontrivial group is at least 2. A maximal abelian subgroup of cardinality 2 consists of the identity element of the group and an involution that do not commute with any other nonidentity element of the group. If $a$ and $b$ are any two elements that do not commute in a finite nonabelian group, then $a$ and $b$ are not in the same maximal abelian subgroup of the group and the subset $\{a,b,a*b\}$ is noncommuting. Therefore, a finite nonabelian group must have at least three maximal abelian subgroups. It is obvious that if $\{\Lambda_1, \Lambda_2, \dots, \Lambda_m\}$ is the collection of all maximal abelian subgroups of a finite nonabelian group $\Gamma$, where $m\geq 3$, then $\Gamma = \bigcup\limits_{i=1}^{m} \Lambda_i$. A finite group $\Gamma$ has only one maximal abelian subgroup if and only if $\Gamma$ is abelian. In this case, the maximal abelian subgroup is the group itself. 

The \textit{commuting graph} $CG(\Gamma, X)$, where $\Gamma$ is a finite group and $X$ is a nonempty subset of $\Gamma$, is a graph whose set of vertices is $X$ and two members $a,b\in X$ are adjacent if $a*b = b*a$ in $\Gamma$~\cite{Ali}. If $X=\Gamma$, then $CG(\Gamma,X)$ is written as $CG(\Gamma)$. Obviously, $CG(\Gamma)$ is connected, since the identity element of any group $\Gamma$ comumutes with all elements of the group. Figure 1 shows the commuting graph of group $SD_{24}$, the semidihedral group with 24 members. The presentation of the group is $SD_{24}=\langle a,b: a^{12}=b^2=1, ba=a^{5}b\rangle$.
\begin{figure}[ht]
	\centering
	% Use \includegraphics to import figures; for example 
	\includegraphics[scale=0.84]{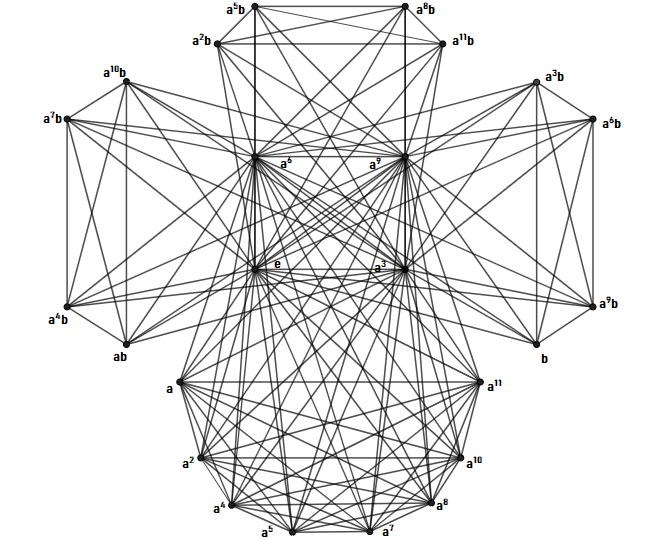}
	\caption{The commuting graph of group $SD_{24}$.\label{fig:sd8}}
\end{figure}

Every two elements of a maximal abelian subgroup of a finite group $\Gamma$ are adjacent in $CG(\Gamma)$. Hence, a maximal abelian subgroup of $\Gamma$ forms a maximal clique in $CG(\Gamma)$, and a maximal abelian subgroup of order 2 forms a pendant edge in $CG(\Gamma)$. Since a finite abelian group has only one maximal abelian subgroup, which is the group itself, the commuting graph the group is a complete graph, and hence its rainbow connection number is 1. If $\Gamma$ is a finite nonabelian group, then $\Gamma$ has more than one maximal commuting subsets and its commuting graph is not a complete graph..
%\begin{lemma}\label{lem:span1}
%Let $\Gamma$ be a finite nonabelian group with $Z(\Gamma)=\{e\}$, $\{C_1, C_2, \dots, C_m\}$ be the collection of all maximal commuting subsets of $\Gamma$ where $m\geq3$, $C'_1=C_1$, and $C'_i=(C_i\setminus \bigcup\limits_{j=1}^{i-1} C_{j})\cup\{e\}$ for every $i\in\{2,3,\dots,m\}$. Then $Amal(K_{|C'_1|},\dots,K_{|C'_m|},e)$ is a spanning subgraph of $CG(\Gamma)$.
%\end{lemma}
%\begin{proof}[Proof]
%Let $\Gamma$ be a finite nonabelian group with $Z(\Gamma)=\{e\}$, $\{C_1, C_2, \dots, C_m\}$ be the collection of all maximal commuting subsets of $\Gamma$ where $m\geq3$, $C'_1=C_1$, and $C'_i=(C_i\setminus \bigcup\limits_{j=1}^{i-1} C_{j})\cup\{e\}$ for every $i\in\{2,3,\dots,m\}$. Clearly, $C'_i\cap C'_j=\{e\}$ for every  $i,j\in\{2,3,\dots,m\}$ where $i\ne j$. Because every two members of $C'_i$ are adjacent in $CG(\Gamma)$, $Amal(K_{|C'_1|},\dots,K_{|C'_m|},e)$ is a subgraph of $CG(\Gamma)$. Since $\Gamma = \bigcup\limits_{i=1}^{m} C_{i}=\bigcup\limits_{i=1}^{m} C'_{i}$, we get $Amal(K_{|C'_1|},\dots,K_{|C'_m|},e)$ is a spanning subgraph of $CG(\Gamma)$.
%\end{proof}

\section{Main results}
Throughout this section, we discuss rainbow connection numbers of the commuting graphs of finite nonabelian groups. If $\Gamma$ is a finite nonabelian group, then $CG(\Gamma)$ is not a complete graph and hence $rc(CG(\Gamma))\geq 2$. We start the discussion with finite nonabelian groups whose commuting graphs have rainbow connection numbers at most 3.

\begin{theorem}\label{Thm:3}
	Let $\Gamma$ be a finite nonabelian group. The rainbow connection number of $CG(\Gamma)$ is at most $3$ if and only if $\Gamma$ has a nontrivial center or $\Gamma$ is a group with a trivial center and has at most three maximal abelian subgroups of order $2$.
\end{theorem}
\begin{proof}[Proof]
	Let $\Gamma$ be a finite nonabelian group with a nontrivial center and $Z(\Gamma)$ be the center of $\Gamma$. Since any member of $Z(\Gamma)$ commutes with all members of $\Gamma$, any member of $Z(\Gamma)$ is adjacent with all vertices in $V(CG(\Gamma))$. Choose any two members $z_1$ and $z_2$ from $Z(\Gamma)$. For every element $x\in \Gamma\setminus Z(\Gamma$), the edge $xz_1$ is colored with color 1 and the edge $xz_2$ is colored with color 2. The edge $z_1z_2$ is colored with color 3. The other edges of $CG(\Gamma)$ are colored with color 1. With this edge coloring, every two distinct elements $x,y \in \Gamma\setminus Z(\Gamma)$ are connected by the path $xz_1z_2y$ which has 3 edges with 3 distinct colors. Hence, $rc(CG(\Gamma))\leq3$.  
	
	Now let $\Gamma$ be a finite nonabelian group with a trivial center and has at most three maximal abelian subgroups of order 2. Let $e$ be the identity element of $\Gamma$, $\mathscr{C}=\{C_1,\dots, C_m\}$ be the collection of all maximal abelian subgroups of $\Gamma$ with $m\geq3$ (since $\Gamma$ is nonabelian), and let $\mathscr{C}$ has exactly $t$ members of order 2, where $0\leq t \leq 3$. Construct a set $\bar{\mathscr{C}} = \{\bar{C}_1, \dots, \bar{C}_t\}$, which is the collection of all maximal abelian subgroups of order 2 and $\hat{C} = (\Gamma\setminus\bigcup\limits_{i=1}^{t}\bar{C}_i)\cup \{e\}$. Write $\bar{C}_i=\{e,\bar{c}_i\}$ for every $i\in \{1, \dots, t\}$ and write $\hat{C} = \{e, \hat{c}_1, \hat{c}_2, \dots, \hat{c}_s\}$ in an order such that $\hat{c}_i$ commutes with $\hat{c}_{i-1}$ or with $\hat{c}_{i+1}$ for every $i\in \{1, \dots, s\}$, where $s=|\Gamma|-t-1$. 
	
	If $t=0$, then $\bar{\mathscr{C}}=\emptyset$, $\hat{C} = \Gamma$, and we can color the edges of $CG(\Gamma)$ with three colors as follows:
	\begin{enumerate}
		\item the edge $e\hat{c}_i$ is colored with color 1 if $i$ is odd or color 2 if $i$ is even for every $i \in \{1,2,\dots,s\}$,
		\item the other edges are colored with color 3.
	\end{enumerate}
	With this coloring, the rainbow paths between two vertices of $CG(\Gamma)$ are as follows:
	\begin{enumerate}
		\item $\hat{c}_i\hat{c}_j$ if $\hat{c}_i$ and $\hat{c}_j$ are adjacent, where $i,j\in \{1, 2, \dots, s\}$,
		\item $\hat{c}_ie\hat{c}_j$ if $\hat{c}_i$ and $\hat{c}_j$ are not adjacent and the parities of $i$ and $j$ are different, where $i,j\in \{1, 2, \dots, s\}$,
		\item if $\hat{c}_i$ and $\hat{c}_j$ are not adjacent and the parities of $i$ and $j$ are the same, where $i,j\in \{1, 2, \dots, s\}$, then the rainbow path between $\hat{c}_i$ and $\hat{c}_j$ is 	$\hat{c}_ie\hat{c}_{j-1}\hat{c}_j$ if $\hat{c}_{j-1}$ is adjacent to $\hat{c}_j$ and $j\ne1$, or $\hat{c}_ie\hat{c}_{j+1}\hat{c}_j$ if $\hat{c}_{j+1}$ is adjacent to $\hat{c}_j$ and $j\ne s$.	 
	\end{enumerate}
	Since every two vertices of $CG(\Gamma)$ are connected by a rainbow path under this edge coloring, we have $rc(CG(\Gamma))\leq 3$.
	
	If $1\leq t \leq 3$, the edge $e\bar{c}_i$ is a pendant edge of $CG(\Gamma)$ for every $i\in \{1, \dots, t\}$ and we can color the edges of $CG(\Gamma)$ with three colors as follows:
	
	\begin{enumerate}
		\item the edge $e\bar{c}_i$ is colored with color $i$ for every $i \in \{1,\dots,t\}$,
		\item the edge $e\hat{c}_i$ is colored with color 1 if $i$ is odd or color 2 if $i$ is even for every $i \in \{1,2,\dots,s\}$,
		\item the other edges are colored with color 3.
	\end{enumerate}
	The rainbow paths between two vertices of $CG(\Gamma)$ under this edge coloring are as follows.
	\begin{enumerate}
		\item The rainbow paths between $\hat{c}_i$ and $\hat{c}_j$, where $i, j\in\{1, 2, \dots, s\}$ and $i\ne j$, are the same as the paths in the case of $t=0$.	
		\item For $i,j\in \{1, \dots, t\}$, where $i\ne j$, the rainbow path between $\bar{c}_i$ and $\bar{c}_j$ is $\bar{c}_ie\bar{c}_j$. 
		\item For $j \in \{1,2,\dots,s\}$, the rainbow path between $\bar{c}_1$ and $\hat{c}_j$ is $\bar{c}_1e\hat{c}_j$ if $j$ is even, $\bar{c}_1e\hat{c}_{j-1}\hat{c}_j$ if $j$ is odd, $\hat{c}_{j-1}$ is adjacent to $\hat{c}_j$, and $j\ne1$, or $\bar{c}_1e\hat{c}_{j+1}\hat{c}_j$ if $j$ is odd, $\hat{c}_{j+1}$ and $\hat{c}_j$ are adjacent, and $j\ne s$.
		
		\item If $t=2$ and $j \in \{1,2,\dots,s\}$, the rainbow path between $\bar{c}_2$ and $\hat{c}_j$ is $\bar{c}_2e\hat{c}_j$ if $j$ is odd, $\bar{c}_2e\hat{c}_{j-1}\hat{c}_j$ if $j$ is even and $\hat{c}_{j-1}$ is adjacent to $\hat{c}_j$, or $\bar{c}_2e\hat{c}_{j+1}\hat{c}_j$ if $j$ is even, $\hat{c}_{j+1}$ is adjacent to $\hat{c}_j$, and $j\ne s$. The rainbow paths between $\bar{c}_1$ and $\hat{c}_j$ are the same as the paths in point 3.	 
		
		\item If $t=3$ and $j \in \{1,2,\dots,s\}$, the rainbow path between $\bar{c}_3$ and $\hat{c}_j$ is $\bar{c}_3e\hat{c}_j$ . The rainbow paths between $\bar{c}_1$ and $\hat{c}_j$, and between $\bar{c}_2$ and $\hat{c}_j$, are the same as the paths in point 3 and 4, respectively.		
	\end{enumerate}
	Since every two distinct vertices of $CG(\Gamma)$ are connected by a rainbow path under this edge coloring, we get $rc(CG(\Gamma))\leq 3$. From the above explanations, we conclude that if $\Gamma$ has a nontrivial center or if $\Gamma$ is a group with a trivial center and has at most three maximal abelian subgroups of order 2, then $rc(CG(\Gamma))\leq 3$. 
	
	Now suppose that $\Gamma$ is a finite nonabelian group with a trivial center and has at least $4$ maximal abelian subgroups of order 2. Therefore, $CG(\Gamma)$ has at least $4$ pendant edges and $rc(CG(\Gamma))\geq 4$. Thus, if $rc(CG(\Gamma))\leq 3$, then $\Gamma$ has a nontrivial center or $\Gamma$ is a group with a trivial center and has at most three maximal abelian subgroups of order 2.  
\end{proof}

Theorem~\ref{Thm:3} characterizes all finite nonabelian groups whose rainbow connection numbers of their commuting graphs do not exceed 3. The theorem states that the rainbow connection number of the commuting graph of a finite nonabelian groups with nontrivial center cannot be greater than 3. In Theorem 4, we discuss a finite nonabelian group with nontrivial center whose rainbow connection number of its inverse graph equals 2. 
\begin{theorem}\label{Thm:4}
	Let $\Gamma$ be a finite nonabelian group with $|Z(\Gamma)|\geq2$ and $\mathscr{C}$ be the collection of all maximal abelian subgroups of $\Gamma$. If $|\mathscr{C}|\leq 2^{|Z(\Gamma)|}$, then $rc(CG(\Gamma))= 2$.
\end{theorem}

\begin{proof}[Proof] 
	Let \(\Gamma\) be a non-Abelian finite group with \(|Z(\Gamma)| \geq 2\) and having \(n\) maximal Abelian subgroups where \(n \geq 3\). Let \(C = \{C_1, \dots, C_n\}\) be the collection of all maximal Abelian subgroups of \(\Gamma\). Since \(\Gamma\) is non-Abelian, \(CG(\Gamma)\) is not a complete graph. Therefore, \(rc(CG(\Gamma)) \geq 2\).

Select any element \(u_i\) from \(C_i \setminus Z(\Gamma)\) for each \(i \in \{1, 2, \dots, n\}\) and form a set \(U = \{u_1, u_2, \dots, u_n\}\). If there exist two elements \(u_i\) and \(u_j\) such that \(u_i = u_j\) for distinct \(i\) and \(j\) in \(\{1, 2, \dots, n\}\) (i.e., \(u_i\) and \(u_j\) are the same non-central element in an intersection of two or more maximal Abelian subgroups), then remove one of the two elements from \(U\). Hence, \(|U| \leq |C|\).

Let \(J\) be an induced subgraph of \(CG(\Gamma)\) with the vertex set \(V(J) = Z(\Gamma) \cup U\). Note that every member of \(Z(\Gamma)\) is adjacent to every other element in \(Z(\Gamma) \cup U\) in \(J\). Suppose \(J\) is edge-colored with two colors. Each \(u_i \in U\) is labeled with an ordered \(|Z(\Gamma)|\)-tuple, denoted by \(K_{u_i} = (k_{u_i1}, k_{u_i2}, \dots, k_{u_i|Z(\Gamma)|})\), where \(k_{u_ij}\) is the color of the edge \(u_iz_j\), with \(z_j\) being a member of \(Z(\Gamma)\), and \(j \in \{1, \dots, |Z(\Gamma)|\}\). Since only two colors are used to color the edges of \(J\), we have \(k_{u_ij} \in \{1, 2\}\) for each \(u_i \in U\) and each \(j \in \{1, \dots, |Z(\Gamma)|\}\). Therefore, the number of distinct ordered \(|Z(\Gamma)|\)-tuples that can be used to label the elements in \(U\) is \(2^{|Z(\Gamma)|}\).

If \(|C| \leq 2^{|Z(\Gamma)|}\), then every two distinct elements in \(U\) that come from different maximal Abelian subgroups can be labeled with different \(|Z(\Gamma)|\)-tuples. If \(u_i\) and \(u_j\) are two distinct elements in \((C_i \cap C_j) \setminus Z(\Gamma)\), the shortest rainbow path in \(J\) connecting these two elements is the edge \(u_iu_j\). For \(u_i \in C_i \setminus C_j\) and \(u_j \in C_j \setminus C_i\), the rainbow path of length 2 in the subgraph \(J\) that connects \(u_i\) and \(u_j\) is \(u_iz_lu_j\), where \(z_l\) is an element of \(Z(\Gamma)\) such that \(k_{u_il} \neq k_{u_jl}\). Since each member of \(U\) is chosen arbitrarily from \(C_i \setminus Z(\Gamma)\) for each \(i \in \{1, 2, \dots, n\}\), it follows that any two elements from different maximal Abelian subgroups are connected by a rainbow path. Note that any two noncommuting elements in \(\Gamma\) do not belong to the same maximal Abelian subgroup. Therefore, if \(|C| \leq 2^{|Z(\Gamma)|}\), the edges of the graph \(CG(\Gamma)\) can be colored with two colors such that any two non-adjacent vertices are connected by a rainbow path of length 2. Consequently, \(rc(CG(\Gamma)) \leq 2\). Since \(rc(CG(\Gamma)) \geq 2\), it follows that \(rc(CG(\Gamma)) = 2\).
\end{proof}

The semidihedral group $SD_{8n}$, which is a group of order $8n$ with $n\geq2$, is an example of a group that meet the condition of Theorem~\ref{Thm:4}. The presentation of the group is $SD_{8n}=\langle a,b: a^{4n}=b^2=e, ba=a^{2n-1}b\rangle$. The set of elements of the group is $\{e,a,a^2,\dots,a^{4n-1}, b,ab,a^2b,\dots,a^{4n-1}b\}$. The center of $SD_{8n}$ is $Z(SD_{8n})=\{e,a^{2n}\}$ if $n$ is even, and $Z(SD_{8n})=\{e,a^n,a^{2n},a^{3n}\}$ if $n$ is odd. For an even $n$, the maximal abelian subgroups of $SD_{8n}$ are $H_i=\{e,a^{2n},a^ib,a^{2n+i}b\}$ for $i\in\{1,2,\dots,3n-1\}$ and $H_{3n}=\{e,a,a^2,\dots,a^{4n-1}\}$. For an odd $n$, the maximal abelian subgroups of $SD_{8n}$ are $H_i=\{e,a^n,a^{2n},a^{3n},a^ib,a^{n+i}b,a^{2n+i}b,a^{3n+i}b\}$ for $i\in\{1,2,\dots,n\}$ and $H_{n+1}=\{e,a,a^2,\dots,a^{4n-1}\}$. If $n=$ 3, $SD_{8n}= SD_{24}$ and has four maximal abelian subgroups. Hence, $rc(CG(SD_{24}))=2$. Figure 2 shows the minimum rainbow coloring of the commuting graph of the group $SD_{24}$ with two colors: color 1 and color 2.

\begin{figure}[ht]
	\centering
	% Use \includegraphics to import figures; for example 
	\includegraphics[scale=0.39]{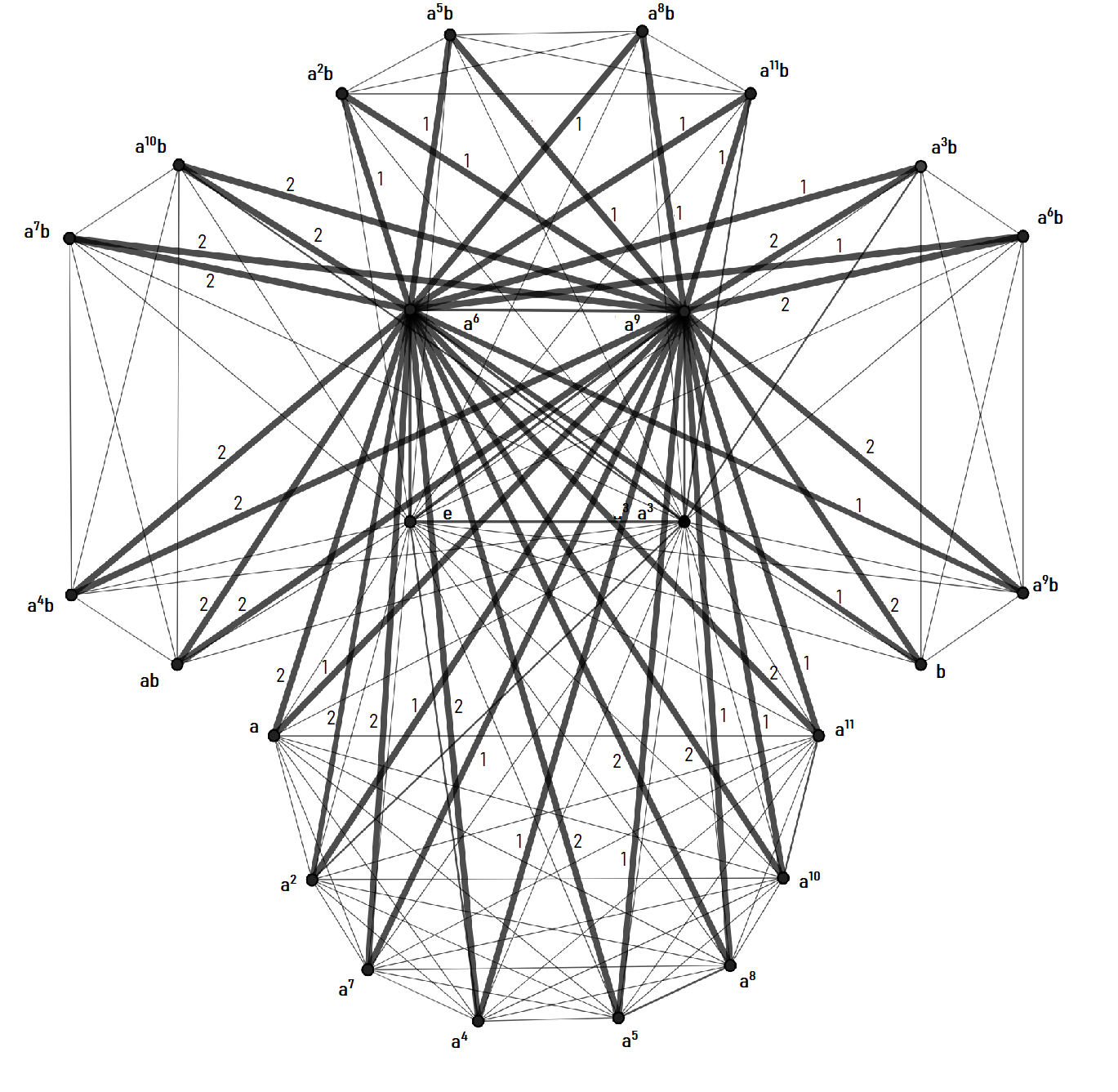}
	\caption{The minimum rainbow coloring of the commuting graph of group $SD_{24}$ with two colors. The edges whose color numbers are not written are colored with color 1.\label{fig:pewarnaansd8n}}
\end{figure}

From Theorem ~\ref{Thm:4}, we directly get the fact that for a finite nonabelian group $\Gamma$ with $|Z(\Gamma)|\geq2$, if $rc(CG(\Gamma))= 3$, then $\Gamma$ has more than $2^{|Z(\Gamma)|}$ maximal abelian subgroups. In order to obtain some finite nonabelian groups whose rainbow connection numbers of their commuting graphs equal 3, we need the following theorem.

\begin{theorem}\label{Thm:5}
	Let \(\Gamma\) be a finite non-Abelian group and \(H \subset \Gamma\) with \(|H| \geq 2\). If \(\Gamma\) has a collection \(\mathscr{T}\) of at least two maximal Abelian subgroups such that the intersection of any two subgroups in \(\mathscr{T}\) is equal to \(H\), and \(X\) is the union of all members of \(\mathscr{T}\), then

\begin{enumerate}
    \item \(rc(CG(\Gamma, X)) = 2\) if and only if \(|\mathscr{T}| \leq 2^{|H|}\);
    \item \(rc(CG(\Gamma, X)) = 3\) if and only if \(|\mathscr{T}| > 2^{|H|}\).
\end{enumerate}

\end{theorem}

\begin{proof}[Proof] 
Suppose $\Gamma$ is a finite non-Abelian group and $H$ is a subset of $\Gamma$ with $|H|\geq 2$. Let $\Gamma$ have a collection of maximal Abelian subgroups $\mathscr{T} = \{\Theta_1, \dots, \Theta_m\}$ with $m \geq 2$, satisfying $H = \Theta_i \cap \Theta_j$ for every distinct $i$ and $j$ in $\{1, \dots, m\}$, and let $X = \bigcup\limits_{i=1}^{m} \Theta_i$. Since $\Gamma$ is a non-Abelian group and $m \geq 2$, there are two vertices in $CG(\Gamma, X)$ that are not adjacent, which implies that $CG(\Gamma, X)$ is not a complete graph. Consequently, $rc(CG(\Gamma, X)) \geq 2$. Any two distinct elements in $\Theta_i$ are adjacent in $CG(\Gamma, X)$ for each $i \in \{1, \dots, m\}$, so every two distinct elements in $H$ are adjacent in $CG(\Gamma, X)$. Since $\Theta_i \cap \Theta_j = H$ for every $i, j \in \{1, \dots, m\}$ with $i \neq j$, each $a \in \Theta_i \setminus H$ is not adjacent to any $b \in \Theta_j \setminus H$ in $CG(\Gamma, X)$ when $i \neq j$. Hence, the shortest paths in $CG(\Gamma, X)$ connecting $a \in \Theta_i \setminus H$ and $b \in \Theta_j \setminus H$ with $i \neq j$ are paths of the form $ah_kb$ where $h_k \in H$ and $k \in \{1, \dots, |H|\}$. These shortest paths have length 2.

Select one element $w_i$ from $\Theta_i \setminus H$ for each $i \in \{1, \dots, m\}$, then form the subset $W = \{w_1, \dots, w_m\}$. Clearly, any two distinct elements in $W$ are not adjacent in $CG(\Gamma, X)$. Consider the subgraph induced by $CG(\Gamma, X)$ with the vertex set $H \cup W$. Denote this subgraph as $J$. It is evident that $J$ is not a complete graph, so $rc(J) \geq 2$. Next, two cases will be considered to determine the rainbow connection number of the graph $J$.

\textbf{Case 1}: For $m \leq 2^{|H|}$.

Let the edges of $J$ be colored with a 2-coloring $p$. For every two elements $h_i$ and $h_j$ in $H$, the edge $h_ih_j$ is colored with color 1. For each $w_i \in W$, $i \in \{1, \dots, m\}$, assign a labeled ordered \textit{tuple}-$|H|$ as $K_{w_i} = (k_{w_{i1}}, k_{w_{i2}}, \dots, k_{w_{i|H|}})$, where $k_{w_{ij}}$ is the color of the edge $w_ih_j$, $h_j \in H$, and $j \in \{1, \dots, |H|\}$. Since the number of colors used is two, $k_{w_{ij}} \in \{1, 2\}$ for each $i \in \{1, \dots, m\}$ and each $j \in \{1, \dots, |H|\}$. Thus, the number of possible ordered \textit{tuples}-$|H|$ for the elements of $W$ is $2^{|H|}$. Note that $|W| = |\mathscr{T}| = m$. If $m \leq 2^{|H|}$, then every two distinct elements of $W$ can be assigned different labeled ordered \textit{tuples}-$|H|$, ensuring that every two elements of $W$ are connected by a rainbow path in $J$. The rainbow path connecting each pair $w_i$ and $w_j$ in $W$ has the form $w_ih_lw_j$, where $h_l$ is an element of $H$ such that $k_{w_{il}} \neq k_{w_{jl}}$. Therefore, we have $rc(J) \leq 2$. Since $rc(J) \geq 2$, it follows that $rc(J) = 2$.

\textbf{Case 2}: For $m > 2^{|H|}$.

As in Case 1, each $w_i \in W$, $i \in \{1, \dots, m\}$, is assigned a labeled ordered \textit{tuple}-$|H|$ as $K_{w_i} = (k_{w_{i1}}, k_{w_{i2}}, \dots, k_{w_{i|H|}})$, where $k_{w_{ij}}$ is the color of the edge $w_ih_j$, $h_j \in H$, and $j \in \{1, \dots, |H|\}$. 

Just as in Case 1, if we color the edges of the graph $J$ with two colors, the number of possible ordered \textit{tuples}-$|H|$ for the elements of $W$ is $2^{|H|}$. Since $m > 2^{|H|}$, there are at least two vertices $w_i$ and $w_j$ in $W$ such that $K_{w_i} = K_{w_j}$. Since $k_{w_{il}} = k_{w_{jl}}$ for each $l \in \{1, \dots, |H|\}$, there is no rainbow path of length 2 connecting $w_i$ and $w_j$ in $J$. Any other paths in $J$ that are not the shortest paths between $w_i$ and $w_j$ have length at least 3, so coloring the edges of $J$ with two colors does not result in a rainbow path in $J$ connecting $w_i$ and $w_j$. Therefore, we have $rc(J) \geq 3$.

Now suppose that each edge $\epsilon$ of $J$ is colored with three colors as follows:
\[
    p* =
    \begin{cases}
        1, & \text{if } \epsilon = h_ih_{i+1} \text{ for } i \in \{1, \dots, |H|-1\} \text{ or } \epsilon = h_iw_i \text{ for } i \in \{1, \dots, |H|\}, \\
        2, & \text{if } \epsilon = h_1w_t \text{ for } t \in \{|H|+1, \dots, m\}, \\
        3, & \text{for all other edges}.
    \end{cases}
\]

With this coloring, the rainbow paths between any two vertices in $J$ are as follows:
\begin{enumerate}
    \item The rainbow path between any two vertices in $H$ is the edge between those two vertices since every two vertices in $H$ are adjacent.
    %\item For each $i \in \{1, \dots, m\}$, the rainbow path between any two vertices in $\Theta_i \setminus H$ is the edge between those two vertices since every two elements of $\Theta_i \setminus H$ are adjacent.
    \item The rainbow path between two vertices $w_i$ and $w_j$ with $i$ and $j$ in $\{1, \dots, |H|\}$ and $i \neq j$ is $w_i h_i w_j$.
    \item The rainbow path between two vertices $w_i$ and $w_j$ with $i$ and $j$ in $\{|H| + 1, \dots, m\}$ and $i \neq j$ is $w_i h_1 h_2 w_j$.
    \item The rainbow path between two vertices $w_i$ and $w_j$ with $i \in \{1, \dots, |H|\}$ and $j \in \{|H| + 1, \dots, m\}$ is $w_i h_1 w_j$.
\end{enumerate}
From the above description, it is clear that with the coloring $p^*$, any two distinct vertices in $J$ are connected by a rainbow path, so $rc(J) \leq 3$. Since $rc(J) \geq 3$, we conclude that $rc(J) = 3$.

If $|\Theta_i| \geq 4$ for at least one $i \in \{1, \dots, m\}$, then there exists another subgraph in $CG(\Gamma, X)$ that is isomorphic to $J$. In this case, to obtain another subgraph isomorphic to $J$, select one element $w'_i$ from $\Theta_i \setminus H$ for each $i \in \{1, \dots, m\}$, with $w'_i \neq w_i$ for at least one $i \in \{1, \dots, m\}$. Then form the subset $W' = \{w'_1, \dots, w'_m\}$. Clearly, any two distinct elements in $W'$ are not adjacent in $CG(\Gamma, X)$. The subgraph of $CG(\Gamma, X)$ induced by $H \cup W'$, call it subgraph $J'$, is isomorphic to $J$.

In $CG(\Gamma, X)$, there are $\prod_{i=1}^{m} |\Theta_i \setminus H|$ subgraphs that are isomorphic to $J$. Since these subgraphs are isomorphic to $J$, the rainbow connection number of each of these subgraphs is the same as $rc(J)$. Any two non-adjacent vertices in $CG(\Gamma, X)$ lie in the same subgraph, either $J$ or a subgraph isomorphic to $J$.

Next, the rainbow connection number $rc(J)$ will be used to determine $rc(CG(\Gamma, X))$. As with determining $rc(J)$, the determination of $rc(CG(\Gamma, X))$ will be divided into two cases.

\textbf{Case 1}: For $m \leq 2^{|H|}$.
Color the edges of $CG(\Gamma, X)$ using the following edge-coloring scheme:
\begin{enumerate}
    \item The subgraph $J$, and all subgraphs isomorphic to $J$, are colored using the minimum rainbow coloring for $J$, which uses two colors.
    \item All other edges in $CG(\Gamma, X)$ are colored with color 1.
\end{enumerate}
Since any two non-adjacent vertices in $CG(\Gamma, X)$ lie in the same subgraph (either the subgraph $J$ or a subgraph isomorphic to $J$), with the above edge-coloring, any two non-adjacent vertices in $CG(\Gamma, X)$ are connected by a rainbow path. Therefore, $rc(CG(\Gamma, X)) \leq 2$. Since $rc(CG(\Gamma, X)) \geq 2$, we conclude that $rc(CG(\Gamma, X)) = 2$.

\textbf{Case 2}: For $m > 2^{|H|}$.
Color the edges of $CG(\Gamma, X)$ using the following edge-coloring scheme:
\begin{enumerate}
    \item The subgraph $J$, and all subgraphs isomorphic to $J$, are colored using the minimum rainbow coloring for $J$, which uses three colors.
    \item All other edges in $CG(\Gamma, X)$ are colored with color 1.
\end{enumerate}
Since any two non-adjacent vertices in $CG(\Gamma, X)$ lie in the same subgraph (either the subgraph $J$ or a subgraph isomorphic to $J$), with the above edge-coloring, any two non-adjacent vertices in $CG(\Gamma, X)$ are connected by a rainbow path. Thus, $rc(CG(\Gamma, X)) \leq 3$. If we attempt to color the edges of $CG(\Gamma, X)$ with fewer than three colors, then there exist two non-adjacent vertices in the subgraph $J$ that are not connected by a rainbow path, implying that $rc(CG(\Gamma, X)) \geq 3$. Therefore, we conclude that $rc(CG(\Gamma, X)) = 3$.

From the above results, several conclusions can be drawn. If $m \leq 2^{|H|}$, then $rc(CG(\Gamma, X)) = 2$. If $m > 2^{|H|}$, then $rc(CG(\Gamma, X)) = 3$. Thus, we can conclude that $rc(CG(\Gamma, X)) = 2$ if and only if $m \leq 2^{|H|}$, and $rc(CG(\Gamma, X)) = 3$ if and only if $m > 2^{|H|}$.

\end{proof}

Using Theorem~\ref{Thm:5}, we obtain some finite nonabelian groups whose rainbow connection numbers of their commuting graphs equal 3.
\begin{theorem}\label{Thm:6}
	Let $\Gamma$ be a finite non-Abelian group with at most three maximal Abelian subgroups of order 2. Let $m$ be an integer with $m \geq 2$, $H \subset \Gamma$ with $|H| \geq 1$, and there exists a collection of maximal Abelian subgroups $\mathscr{T} = \{\Theta_1, \dots, \Theta_m\}$ of the group $\Gamma$ such that $\Theta_i \cap \Theta_j = H$ for $i \neq j$. If $m > 2^{|H|}$, then $rc(CG(\Gamma)) = 3$.
\end{theorem}

\begin{proof}[Proof]
	Let $|Z(\Gamma)|\geq2$, $m > 2^{|H|}$, and $X = \bigcup\limits_{i=1}^{m} \Theta_{i}$. According to Theorem~\ref{Thm:3}, $rc(CG(\Gamma))\leq 3$. Since $|Z(\Gamma)|\geq2$, we get $|H|\geq 2$. According to Theorem~\ref{Thm:5}, $rc(CG(\Gamma, X))= 3$. It is obvious that $CG(\Gamma, X)$ is a subgraph of $CG(\Gamma)$. The length of any path $P$ in $CG(\Gamma)$, whose $E(P)$ is not the subset of $E(CG(\Gamma, X))$, that connects any two nonadjacent vertices $a\in \Theta_i\setminus H$ and $b\in \Theta_j\setminus H$ for $i\ne j$ is at least $3$. Therefore, we need an edge coloring with at least 3 colors to make $CG(\Gamma)$ rainbow-connected, and hence $rc(CG(\Gamma))\geq 3$. Since $rc(CG(\Gamma))\leq 3$, we get $rc(CG(\Gamma))= 3$.
	
	Now let $|Z(\Gamma)|=1$, $m > 2^{|H|}$, and $\Gamma$ has at most $3$ maximal abelian subgroups of order $2$. According to Theorem~\ref{Thm:3}, $rc(CG(\Gamma))\leq 3$. Let $X = \bigcup\limits_{i=1}^{m} \Theta_{i}$ and $|H|=1$. Since $|Z(\Gamma)|= |H| = 1$, we get $Z(\Gamma) = H$ and $m \geq 3$. Suppose that $rc(CG(\Gamma, X))=2$. Choose three non-identity elements $c_1\in \Theta_1$, $c_2\in \Theta_2$, and $c_3\in \Theta_3$. Clearly, every pair of the three elements are not adjacent in $CG(\Gamma,X)$, and also in $CG(\Gamma,)$ . Since $\Theta_i\cap \Theta_j = \{e\}$ for every $i,j\in \{1, \dots, m\}$, where $i\ne j$, the only paths of length 2 in $CG(\Gamma,X)$, and also in $CG(\Gamma)$, connecting every pair of the three elements are $c_1ec_2$, $c_1ec_3$, and $c_2ec_3$. Without loss of generality, if we color the edge $c_1e$ with color 1, the edge $c_2e$ with color 2, and the edge $c_3e$ with color 2, then $c_1ec_2$ is the rainbow path connecting $c_1$ and $c_2$, and  $c_1ec_3$ is the rainbow path connecting $c_1$ and $c_3$. But there is no rainbow path connecting $c_2$ and $c_3$. Hence, $rc(CG(\Gamma,X))$, and also $rc(CG(\Gamma))$, must be at least 3. Since $rc(CG(\Gamma))\leq 3$, we get $rc(CG(\Gamma))= 3$. If $|H|\geq2$, using a similar proof as in the case of $|Z(\Gamma)|\geq2$, we get $rc(CG(\Gamma))\geq 3$. Since $rc(CG(\Gamma))\leq 3$, we get $rc(CG(\Gamma))= 3$.
\end{proof}	

For a finite nonabelian group $\Gamma$ with a nontrivial center which the intersection of every pair of its maximal abelian subgroups equals its center, Theorem~\ref{Thm:5} immediately gives us the following corollary.

\begin{corollary}\label{Corr:1}
	Let $\Gamma$ be a finite nonabelian group with a nontrivial center $Z(\Gamma)$, $\mathscr{C}$ be the collection of all maximal abelian subgroups of $\Gamma$, and the intersection of every two members of $\mathscr{C}$ equals $Z(\Gamma)$.
	\begin{enumerate}
		\item $rc(CG(\Gamma))=2$ if and only if $|\mathscr{C}|\leq 2^{|Z(\Gamma)|}$,
		\item $rc(CG(\Gamma))=3$ if and only if $|\mathscr{C}|> 2^{|Z(\Gamma)|}$.
	\end{enumerate}
\end{corollary}

The dihedral group $D_6$ and the alternating group $A_4$ are examples of groups that meet the condition of Corollary~\ref{Corr:1}. The presentation of $D_6$ is $D_6 = \langle r, s: r^3 = s^2 = e, srs = r^{-1}\rangle$, where $e$ is the identity element of the group. This group has four maximal abelian subgroups, which are $\{e, r, r^2\}$, $\{e, s\}$, $\{e, rs\}$, and $\{e, r^2s\}$. There are three maximal abelian subgroups of order 2 in this group. The center of this group is $Z(D_6)=\{e\}$ and the intersection of every two distinct maximal abelian subgroups is $Z(D_6)$. The center of group $A_4$ is $Z(A_4)=\{(1)\}$, which is the identity permutation. This group has five maximal abelian subgroups. The maximal abelian subgroups of $A_4$ are $\{(1), (123), (132)\}$, $\{(1), (124), (142)\}$, $\{(1), (134), (143)\}$, $\{(1), (234), (243)\}$, and $\{(1), (12)(34), (13)(24), (14)(23)\}$. This group has no maximal abelian subgroup of order 2. The intersection of every two distinct maximal abelian subgroups of $A_4$ is $Z(A_4)$. According to Corollary~\ref{Corr:1}, $rc(CGD_6))=rc(CG(A_4))=3$. Figure 3 shows the minimum rainbow colorings of $CG(A_4)$ and $CG(D_6)$ with three colors: color 1, color 2, and color 3.
\begin{figure}[ht]
	\centering
	% Use \includegraphics to import figures; for example 
	\includegraphics[scale=0.9]{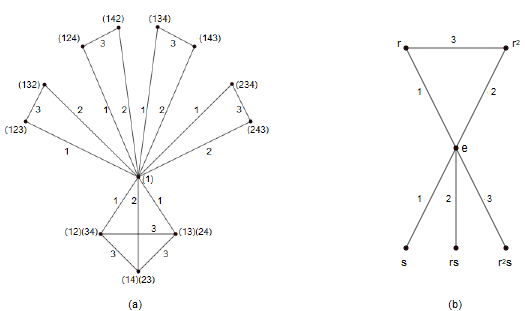}
	\caption{(a) A minimum rainbow coloring of $CG(A_4)$, (b) A minimum rainbow coloring of $CG(D_6)$.\label{fig:a4}}
\end{figure}

An example of a finite nonabelian group that meets the condition of Theorem~\ref{Thm:6} is the generalized quaternion group $Q_{4n}$, where $n\geq2$, with presentation $Q_{4n} = \langle a, b: a^n = b^2, a^{2n}=e, b^{-1}ab = a^{-1}\rangle $. The center of the group is $Z(Q_{4n})=\{e,a^n\}$. Hence, $|Z(Q_{4n})|=2$. The maximal abelian subgroups of the group are $H_i = \{e, a^n, a^ib, a^{n+i}b\}$ for $i \in \{0,1,\dots,n-1\}$ and $H_n = \{e, a, a^2, \dots, a^{2n-1}\}$. Therefore, the group has $n+1$ maximal abelian subgroups and it is obvious that $Z(\Gamma)=H_i\cap H_j$ for $i,j\in \{0,2,\dots,n\}$ where $i\ne j$. If $n\geq 4$, then $Q_{4n}$ has at least 5 maximal abelian subgroups. Hence, according to Theorem~\ref{Thm:6}, $rc(CG(Q_{4n}))=3$ if $n\geq4$. Figure 4 shows a minimum rainbow coloring of the group $Q_{16}$ with three colors: color 1, color 2, and color 3.
\begin{figure}[ht]
	\centering
	% Use \includegraphics to import figures; for example 
	\includegraphics[scale=0.8]{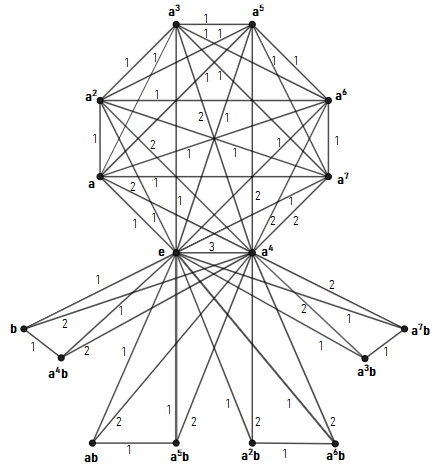}
	\caption{A minimum rainbow coloring of the commuting graph of group $Q_{16}$ with three colors. 
		\label{fig:pewarnaansd8n}}
\end{figure}

%Now consider a finite nonabelian group $\Gamma$ with a nontrivial center. If the intersection any two maximal abelian subgroups of the group equals its center, we will also get an exact value of $rc(CG(\Gamma))$.
%, and $rc(CG(SD_{8n}))=3$ if $n$ is even or $n$ is an odd integer greater than 4 (according to Theorem~\ref{Thm:4}). 
So far we have discussed the finite nonabelian groups whose rainbow connection numbers of their commuting graphs are less than or equal to 3. In Theorem~\ref{Thm:7}, we characterize a finite nonabelian group whose commuting graph has rainbow connection number at least 4.
%rainbow connection number of a finite nonabelian group with a trivial center which has at most three maximal abelian subgroups of order 2. In the following corollary, we determine $rc(CG(\Gamma))$ if $\Gamma$ is a finite nonabelian group with a trivial center which has at least four maximal abelian subgroups of order 2.
\begin{theorem}\label{Thm:7}
	For an integer $t\geq 4$, $rc(CG(\Gamma))=t$ if and only if $\Gamma$ is a finite nonabelian group with a trivial center and has exactly $t$ maximal abelian subgroups of order 2.
	%Let $\Gamma$ be a finite nonabelian group with trivial center and $C=\{C_1,\dots,C_m\}$ be the collection of all maximal abelian subgroups of $\Gamma$ with $m\geq3$. If $C$ has exactly $t$ members of cardinality 2, where $4\leq t \leq m$, then $rc(CG(\Gamma))=t$.
\end{theorem}
\begin{proof}
	Let $t\geq 4$ be an integer and $\Gamma$ be a finite nonabelian group with a trivial center and has exactly $t$ maximal abelian subgroups of order 2. Let $e$ be the identity element of $\Gamma$ and $\mathscr{C}=\{C_1,\dots, C_m\}$ be the collection of all maximal abelian subgroups of $\Gamma$, where $m\geq4$. Form the sets $\bar{\mathscr{C}}$ and $\hat{C}$ as in the proof of Theorem~\ref{Thm:3}. Recall that since $|\bar{\mathscr{C}}|=t$, the commuting graph $CG(\Gamma)$ has exactly $t$ pendant edges. Therefore, $rc(CG(\Gamma))\geq t$. Next, we color the edges of $CG(\Gamma)$ with $t$ colors as follows:
	\begin{enumerate}
		\item the edge $e\bar{c}_i$ is colored with color $i$ for every $i \in \{1,2,\dots,t\}$,
		\item the edge $e\hat{c}_i$ is colored with color 1 if $i$ is odd or color 2 if $i$ is even for every $i \in \{1,2,\dots,s\}$,
		\item the other edges are colored with color 3.
	\end{enumerate}
	The rainbow paths of $CG(\Gamma)$ under this edge coloring are as follows.
	\begin{enumerate}
		\item For $i,j \in \{1,2,\dots,t\}$, where $i\ne j$, the rainbow path between $\bar{c}_i$ and $\bar{c}_j$ is $\bar{c}_ie\bar{c}_j$.
		\item For every $j\in\{1, 2, \dots, s\}$, the rainbow paths between $\bar{c}_1$ and $\hat{c}_j$, and between $\bar{c}_2$ and $\hat{c}_j$, are the same as the paths in the case of $t=2$ in Theorem~\ref{Thm:3}.
		\item For every $i \in \{3,\dots,t\}$ and every $j \in \{1,2,\dots,s\}$, the rainbow path between $\bar{c}_i$ and $\hat{c}_j$ is $\bar{c}_ie\hat{c}_j$.
		\item The rainbow paths between $\hat{c}_i$ and $\hat{c}_j$, where $i, j\in\{1, 2, \dots, s\}$, are the same as the paths between $\hat{c}_i$ and $\hat{c}_j$ in the case of $t=0$ in Theorem~\ref{Thm:3}.		 
	\end{enumerate}
	Since every two distinct vertices of $CG(\Gamma)$ are connected by a rainbow path under this edge coloring, we get $rc(CG(\Gamma))\leq t$. Since $rc(CG(\Gamma))\geq t$, we get $rc(CG(\Gamma))= t$, where $t\geq4$. 
	
	Now let $rc(CG(\Gamma))= t$, where $t\geq4$. According to Theorem~\ref{Thm:3}, $\Gamma$ is a finite nonabelian group with a trivial center and has exactly $s$ maximal abelian subgroups of order 2, where $s\geq4$. Hence, $CG(\Gamma)$ has exactly $s$ pendant edges and $s\leq t$. According the previous result, we can color the edges of $CG(\Gamma)$ by $s$ colors such that every two distinct vertices of $CG(\Gamma)$ are connected by a rainbow path. Therefore, $rc(CG(\Gamma))\leq s$. Since $rc(CG(\Gamma))= t$ and $s\leq t$, we get $s=t$.
\end{proof}

All results above show that the rainbow connection number of the commuting graph of a finite nonabelian group with nontrivial center is related to the number of maximal abelian subgroups of the group. Now recall that a maximal abelian subgroup of order 2 of a finite nonabelian group $\Gamma$ consists of the identity element of $\Gamma$ and an involution that do not commute with any other nonidentity element of $\Gamma$. Hence, for a natural number $t$, $\Gamma$ has exactly $t$ maximal abelian subgroups of order 2 if and only if $\Gamma$ has exactly $t$ involutions that do not commute with any other nonidentity element of the group. Thus, Theorem~\ref{Thm:7} gives us the following corollary.

\begin{corollary}
	For a natural number $t\geq 4$, $rc(CG(\Gamma))=t$ if and only if $\Gamma$ is a finite nonabelian group with a trivial center and has exactly $t$ involutions that do not commute with any other nonidentity element of $\Gamma$.
\end{corollary}	

According to Corollary 2, if $rc(CG(\Gamma))\geq 4$ for a finite group $\Gamma$, then $\Gamma$ must be a group with a trivial center and has $rc(CG(\Gamma))$ involutions  that do not commute with any other nonidentity element of $\Gamma$. If $\Gamma$ is a finite group with a trivial center and $rc(CG(\Gamma))< 4$, then we can be sure that $\Gamma$ has less than four involutions  that do not commute with any other nonidentity element of $\Gamma$.
%\section*{Acknowledgements}

%{\bf Acknowledgement.} Acknowledgements could be placed at the end
%of the text but precede the references.

\bibliographystyle{amsplain}

\end{document}